\documentclass{amsart}
\usepackage{geometry}
\usepackage{amsfonts,amsmath,amsthm,amssymb,dsfont,mathrsfs}
\newtheorem{theorem}{Theorem}
\newtheorem{lemma}{Lemma}
\newtheorem{corollary}{Corollary}

\newtheorem{remark}{Remark}
\newtheorem{definition}{Definition}
\newtheorem{prop}{Proposition}
\begin{document}
\renewcommand{\refname}{References}
\begin{center}
	\textbf{An extension of the standard multifractional Brownian motion}\vspace{0.3cm}
\end{center}
\begin{center}
	\textsl{M. Ait Ouahra\footnote {Corresponding author: ouahra@gmail.com},  H. Ouahhabi $^2$,  A. Sghir  $^1$ and M. Mellouk $^3$}\\\vspace*{0.1cm}
\end{center}
\begin{center}
	\textit{$^{1}$Mohammed First University. Faculty of Sciences Oujda. Department of Mathematics. Stochastic and Deterministic Modelling Laboratory. B.P. 717. Morocco.\\}	\vspace*{0.3cm}
		 \textit{$^2$Department of Statistics.
		College of Business and Economics. United Arab Emirates
		University.}\\	\vspace*{0.3cm}
	 \textit{$^3$ MAP5, CNRS UMR 8145. University of Paris. 45, Street of Saints-Pères 75270. Paris Cedex 6. France.}
	\end{center} 
\vspace*{0.2cm}
\begin{center}
	\textbf{Abstract}\vspace*{0.1cm}
\end{center}
\hspace*{0.3cm}In this paper, firstly, we generalize the definition of the   bifractional Brownian motion $B^{H,K}:=\Big(B^{H,K}\;;\;t\geq 0\Big)$, with parameters $H\in(0,1)$ and $K\in(0,1]$, to the case where  $H$ is no longer a constant, but a function $H(.)$ of the time index $t$ of the process. We denote this new  process by $B^{H(.),K}$.  Secondly, we study its time regularities, the local asymptotic self-similarity and the long-range dependence properties.
\section{Introduction}
In recent years, the famous  fractional Brownian motion $B^{H}:=\Big(B^{H}_{t}\;;\; t\geq 0\Big)$, (fBm for short), with Hurst parameter $H\in(0,1)$, has 
considerable interest due to its applications in various
scientific areas including:  telecommunications, finance, turbulence and image
processing, (see for examples:  Addison and  Ndumu \cite{AN}, Cheridito
\cite{Ch}, Comegna et \emph{al.} \cite{Cal}, Samorodnitsky and  Taqqu \cite{ST} and Taqqu \cite{MTa}). The fBm  was firstly introduced by Kolmogorov \cite{Ko}, and was later made popular by Mandelbrot and Van Ness \cite{MVN}. It is
the only centered and self-similar Gaussian process with stationary increments
 and covariance function:
 $$R^{H}(t,s):=\mathbb{E}\Big(B^H_tB^H_s\Big)=\frac{1}{2}\Big(t^{2H}+s^{2H}-|t-s|^{2H}\Big),\ \ \forall s,t\geq 0.$$
For small increments, in models such as turbulence, fBm seems a good model but it is inadequate for  large increments. For this reason, Houdr\'{e} and Villa
\cite{HV} have explored the existence of a Gaussian process which preserve some of the
properties of the fBm such as self-similarity and stationarity of small increments, and can enlarge modelling tool kit. This processes, denoted by $B^{H,K}:=\Big(B^{H,K}_{t}\;;\; t\geq 0\Big)$, is called the bifractional Brownian motion, (bfBm for short),
with parameters $H\in (0,1)$ and $K\in (0,1]$ and covariance function:
$$R^{H,K}(t,s):=\mathbb{E}\Big(B^{H,K}_tB^{H,K}_s\Big)=\frac{1}{2^{K}}\bigg[\Big(t^{2H}+s^{2H}\bigg)^K-|t-s|^{2HK}\bigg],\ \ \forall t,s\geq 0.$$
For large details about bfBm, we refer  to   \cite{ET}, \cite{HV}, \cite{LN} and \cite {RT}.\newpage
The increments of $B^{H,K}$ are  only independents
in the case of the standard Brownian motion: $(H=\frac{1}{2},K=1)$, and  they are
not stationary for any $K\in]0,1[$, except the case of the
fBm: $(K=1)$, however, $B^{H,K}$ is  quasi-helix in the
sense of Kahane \cite{Ka}:
\begin{equation}
2^{-K}|t-s|^{2HK}\leq\mathbb{E}\Big(B^{H,K}_{t}-B^{H,K}_{s}\Big)^{2}\leq
2^{1-K}|t-s|^{2HK},\hspace{0.3cm} \forall s,t\geq 0.
\end{equation}
Moreover, according to \cite{HV}, if we put:
$\sigma^{2}_{\varepsilon}(t):=\mathbb{E}\Big(B^{H,K}_{t+\varepsilon}-B^{H,K}_{t}\Big)^{2},$
then
$$\lim_{\varepsilon\rightarrow0}\frac{\sigma^{2}_{\varepsilon}(t)}{\varepsilon^{2HK}}=2^{1-K},
\;\;t>0,$$
therefore, the small increments of $B^{H,K}$   are approximately
stationary.
For the large increments, Maejima and Tudor \cite{MT} have proved that, when
$h\rightarrow+\infty$, the sequence of increments process:
$$\left(B_{t+h}^{H,K}-B_{h}^{H,K}\;;\;t\geq0\right)$$ converges modulo a
constant, in the sense of the finite dimensional
distributions, to the fBm $\left(B_{t}^{HK}\;;\;t\geq0\right)$ with Hurst parameter $HK$. This result can be interpreted like the bfBm has stationary
increments for large increments.
The key ingredient used in \cite{MT} is a
decomposition in law  of the bfBm presented by Lei and Nualart \cite{LN} as
follows:\\ Let $W:=\Big(W_{\theta}\;;\;\theta\geq0\Big)$ be a standard Brownian
motion independent of $B^{H,K}$. For any $K\in(0,1)$, let
$X^{K}:=\Big(X^{K}_{t}\;;\;t\geq0\Big)$  the centred Gaussian
process defined by:
$$X^{K}_{t}:=\int_{0}^{+\infty}\Big(1-e^{-\theta t}\Big)\theta^{-\frac{(1+K)}{2}}dW_{\theta},$$
with the covariance function:
$$\mathbb{E}\Big(X^{K}_tX^{K}_s\Big)=\frac{\Gamma(1-K)}{K}\Big[t^{K}+s^{K}-(t+s)^{K}\Big],\ \ \forall t,s\geq 0.$$
$\Gamma$  is the well known Gamma function.\\
The authors in \cite{LN}  showed by setting: $\displaystyle X^{H,K}_{t}:=X^{K}_{t^{2H}},$
that:
\begin{equation}
\left(C_{1}(K)X^{H,K}_{t}+B^{H,K}_{t}\hspace{0.1cm};\hspace{0.1cm}t\geq0\right)
\hspace{0.2cm}\underline{\underline{d}}\hspace{0.2cm}\left(C_{2}(K)B^{HK}_{t}
\hspace{0.1cm};\hspace{0.1cm}t\geq0\right),
\end{equation}
where $C_{1}(K)=\sqrt{\frac{2^{-K}K}{\Gamma(1-K)}}$,
$C_{2}(K)=2^{\frac{(1-K)}{2}}$ and $\underline{\underline{d}}$ means equality of all finite dimensional distributions. The second application of (2) given in \cite{MT} is that the long-range dependence, (LRD for short), of the process $B^{H,K}$ depends on the value of
the product $HK$:
$$\bullet\;\mbox{Long-memory: for every}\; a\in \mathbb{N}:\ \ \sum_{n\geq 0}\mbox{cor}_{B^{H,K}}(a,a+n)=+\infty,\ \ \mbox{if}\ \ 2HK>1,$$ 
$$\bullet\;\mbox{Short-memory: for every}\; a\in \mathbb{N}:\ \ \sum_{n\geq 0}\mbox{cor}_{B^{H,K}}(a,a+n)<+\infty,\ \ \mbox{if}\;2HK\le1.$$ 
where $$\mbox{cor}_{B^{H,K}}(a,a+n):=\mathbb{E}\Big[\Big(B^{H,K}_{a+1}-B^{H,K}_{a}\Big)\Big(B^{H,K}_{a+n+1}-B^{H,K}_{a+n}\Big)\Big].$$
This result was appeared also in    Remark 7 by  Russo and Tudor \cite{RT} .\newpage
Now, we are ready to introduce our new process: Since the model of the fBm $B^H$ may
be restrictive for different phenomena  due to the fact that all its interesting properties
are governed by the Hurst parameter $H$, this gave the motivation to Benassi et \emph{al}. \cite{Beal} and  L\'{e}vy-Véhel and Peltier \cite{LP} to introduce, independently, a new
model to generalize the fBm: It's the multifractional Brownian motion, (mBm for short). Contrarily to the fBm, the almost sure  H\"{o}lder exponent of the mBm is allowed to vary along the trajectory, a useful feature when one needs to model processes whose regularity evolves in time, such as Internet traffic or images. The definition of the mBm in
 \cite{LP} is based on the moving average representation of the fBm, where the constant Hurst parameter $H$ is substituted by a functional $H(.)$ as follows:
 $$\widetilde{B}^{H(t)}_t=\frac{1}{\Gamma\Big(H(t)+\frac{1}{2}\Big)}\bigg(\int_{_\infty}^{0} \bigg[(t-u)^{H(t)-\frac{1}{2}}-(-u)^{H(t)-\frac{1}{2}}\bigg]W(du)$$
 $$+\int_{0}^{t}(t-u)^{H(t)-\frac{1}{2}}W(du)\bigg),\ \ t\geq0,$$
 where $H(.):[0,\infty)\mapsto [\mu,\nu]\subset (0,1)$ is a  H\"{o}lder continuous function of exponent $\beta>0$, and $W$ is a standard Brownian motion on $\mathbb{R}$.\\
 The authors  in \cite{Beal} defined the mBm by means of the harmonisable representation of the fBm as follows:
 $$\widehat{B}^{H(t)}_t=\int_{\mathbb{R}}\frac{e^{it\xi-1}}{|\xi|^{H(t)+\frac{1}{2}}}\widehat{W}(d\xi),\ \ t\geq0,$$
 where $\widehat{W}(\xi)$ is the Fourier transform of the series representation of white noise with respect to an orthonormal basis of $L^{2}(\mathbb{R})$. From these definitions, it's easy to that mBm is a zero mean Gaussian processes whose increments are in general neither independents nor stationary. 
 It is proved by Cohen  \cite{Co} that the two representations 
 of mBm are equivalent, up to a multiplicative deterministic function. This function is explicitly given by Boufoussi et \emph{al.}  \cite{Boal}. Moreover, in Ayache et \emph{al}. \cite{Aal},  the covariance
 function of the standard mBm: (i.e. the variance a time 1 is 1), has been deduced from its harmonisable representation as follows:
 $$\mathbb{E}\Big(B^{H(t)}_tB^{H(s)}_s\Big)=D\Big(H(t),H(s)\Big)\bigg[t^{H(t)+H(s)}+s^{H(t)+H(s)}-|t-s|^{H(t)+H(s)}\bigg],$$
 where
  \[D(x,y):=\frac{\sqrt{(\Gamma(2x+1)\Gamma(2y+1))\sin(\pi x)\sin(\pi y)}}{2\Gamma(x+y+1)\sin\Big(\frac{\pi(x+y)}{2}\Big)}.\]
 Clearly, if $H(.)\equiv H$ a constant in $(0,1)$, $\displaystyle D(H,H)=\frac{1}{2}$, and we find the covariance function of the fBm $B^H$, the zero mean Gaussian process with stationary increments.\\
 In the same spirit as \cite{Beal} and \cite{LP}, since all the properties of the  bfBm $B^{H,K}$ is governed by the unique number $HK,$ we introduce in this note a generalization of $B^{H,K}$, by substituting to the parameter $H$ in the covariance function  $R^{H,K}$,  a H\"{o}lder function  $H(.):[0,\infty)\mapsto [\mu,\nu]\subset (0,1)$ with exponent $\beta>0.$ More precisely:
\begin{definition} We define a new centred  Gaussian process, starting from zero, and denoted by  $B^{H(.),K}:=\Big(B^{H(t),K}_t\;;\;t\geq0\Big)$,  by the covariance function:
$$
	R^{H(.),K}(t,s):=\Big(D(H(t),H(s)\Big)^K\bigg[\Big(t^{H(t)+H(s)}+s^{H(t)+H(s)}\Big)^{K}-|t-s|^{(H(t)+H(s))K}\bigg].
$$
\end{definition} 
\begin{remark}
Clearly, when $K=1,$ $B^{H(.),K}$ is a standard mBm. When $H(.)\equiv H$ a constant in $(0,1)$, $B^{H(.),K}$ is a bfBm with parameters $H\in (0,1)$ and $K\in (0,1]$.
\end{remark}
\section{The existence of $B^{H(.),K}$}
In this section we prove the existence of our process by using the same argument used in
\cite{HV}  for the bfBm.
\begin{prop}
For any $K\in(0,1]$ and $H(.):[0,\infty)\mapsto [\mu,\nu]\subset (0,1)$ a
H\"{o}lder continuous function, the covariance function $R^{H(.),K}$ appeared in Definition 1
is positive-definite.
\end{prop}
\begin{proof}We assume that $K\in(0,1)$ since the special case  $K=1$ is evident. We use the following identity:
  $$t^{K}=\frac{K}{\Gamma(1-K)}\int^{\infty}_{0}(1-e^{-tx})x^{-1-K}dx,\ \ \forall t\geq0,$$
 where $\Gamma$ is the gamma function. 
 For any $c_1,...,c_n \in
  \mathbb{R},$ we have:
\begin{eqnarray*}
&&\sum^{n}_{i=1}\sum^{n}_{j=1} c_i c_j R^{H(.),K}(t_i,t_j)\\
&=&\frac{K}{\Gamma(1-K)}\int^{\infty}_{0}\sum^{n}_{i=1}\sum^{n}_{j=1} c_i c_j \bigg[-e^{-xD(H(t_i),H(t_j))\Big(t_i^{H(t_i)+H(t_j)}+t_j^{H(t_i)+H(t_j)}\Big)}\\
&&\hspace*{1cm}+e^{-xD(H(t_i),H(t_j))\Big(|t_i-t_j|^{H(t_i)+H(t_j)}\Big)}\bigg]x^{-1-K}dx\\
&=&\frac{K}{\Gamma(1-K)}\int^{\infty}_{0}\sum^{n}_{i=1}\sum^{n}_{j=1} c_i c_j
e^{-xD(H(t_i),H(t_j))\Big(t_i^{H(t_i)+H(t_j)}+t_j^{H(t_i)+H(t_j)}\Big)}\\
&&\hspace*{1cm} \times\bigg[e^{xD(H(t_i),H(t_j))\Big(t_i^{H(t_i)+H(t_j)}+t_j^{H(t_i)+H(t_j)}-|t_i-t_j|^{H(t_i)+H(t_j)}\Big)}-1\bigg]x^{-1-K} dx.
\end{eqnarray*}
We know by  \cite{Aal} that
$D(H(t),H(s))\Big(t^{H(t)+H(s)}+s^{H(t)+H(s)}-|t-s|^{H(t)+H(s)}\Big)$ is
positive-definite, then so is:
$$e^{xD(H(t),H(s))\Big(t^{H(t)+H(t)}+s^{H(t)+H(t)}-|t-s|^{H(t)+H(t)}\Big)}-1,\ \ \forall x\geq0,$$
which gives the proof of the proposition. 
\end{proof}
\begin{remark}
It's easy to see the following link between the covariance function of our process $B^{H(.),K}$, the mBm $B^{H(.)K}$ and a transformation of the process $X^K$:
$$\frac{KD\Big(H(t),H(s)\Big)^K}{\Gamma(1-K)}\;cov\Big(X^K_{t^{H(t)+H(s)}},X^K_{s^{H(t)+H(s)}}\Big)+cov\Big(B^{H(t),K}_t,B^{H(s),K}_s\Big)$$
$$=\frac{KD\Big(H(t),H(s)\Big)^K}{D\Big(H(t)K,H(s)K\Big)}\;cov\Big(B^{H(t)K}_t,B^{H(s)K}_s\Big),\hspace*{0.5cm} \forall s,t\geq 0.\hspace*{2.4cm}$$
Therefore,  under the assumption of independence, when $H(.)\equiv H$ a constant in $(0,1)$, we find easily (2) the decomposition in law of the bfBm. However, in the functional case $H(.)$,  since we can not separate the variables $s$ and $t$ in $D\Big(H(t),H(s)\Big)$, then we cannot deduce a decomposition in law of our process $B^{H(.),K}$.
\end{remark}
\section{Regularities of the trajectories of $B^{H(.),K}$}
 In this section, we deal with the regularities of the trajectories of  $B^{H(.),K}$. We follows the same method used  in the case of the mBm, (see \cite{Boal}).  For this, we need  the following regularity of the bfBm $B^{H,K}$ with respect to the constant parameter $H$. We use (2) the decomposition in law of the bfBm $B^{H,K}$. 
\begin{prop}
\label{HH'7} Let $[a,b]\subset [0,\infty)$ and $[\alpha,\gamma]\subset (0,1],$ and consider $B^{H,K}$ a bfBm with parameters $H\in[\alpha,\gamma]$ and $K\in[0,1].$ Then, there exists a finite positive constant $C(\alpha,\gamma,K)$ such that, for all $H,H'\in[\alpha,\gamma],$ we have:
$$
\sup_{t\in[a,b]}\mathbb{E}\Big(B^{H,K}_{t}-B^{H',K}_{t}\Big)^2\leq C(\alpha,\gamma,K)|H-H'|^{2}.
$$
\end{prop}
\begin{proof}
Using (2) and the elementary inequality: $(a-b)^2\leq 2a^2+ 2b^2$, we
obtain:
\begin{eqnarray*}
\mathbb{E}\Big(B^{H,K}_{t}-B^{H',K}_{t}\Big)^2\leq  2 C^{2}_2(K)\mathbb{E}\Big(B^{HK}_{t}-B^{H'K}_{t}\Big)^2+2 C^{2}_1(K)\mathbb{E}\Big(X^{H,K}_{t}-X^{H',K}_{t}\Big)^2.
\end{eqnarray*}
In view of Lemma 3.1 in \cite{Boal}, (see also \cite{LP}),
we know that:
\begin{eqnarray}\label{BHK7}\mathbb{E}\Big(B^{HK}_{t}-B^{H'K}_{t}\Big)^2\leq C_1(\alpha,\gamma,K)
|H-H'|^2.
\end{eqnarray} where
\[C_1(\alpha,\gamma,K)=4\sup_{t\in[a,b]}\bigg(\ \int^{1}_{0} \frac{1-\cos(t\theta)}{\theta^{2\gamma+1}} (\log(\theta))^2 dx +\int^{\infty}_{1}\frac{1}{\theta^{2\alpha+1}}(\log(\theta))^2d\theta\bigg)<+\infty.\]
Now, let us  deal with the process $X^{H,K}$. We have by the It\^{o}'s isometry:
$$\mathbb{E}\Big(X^{H,K}_t-X^{H',K}_t\Big)^2=\int^{\infty}_{0}\Big(e^{-\theta t^{2H'}}-e^{-\theta t^{2H}}\Big)^2 \theta^{-(1+K)}d\theta.$$
Without loss of generality, we suppose that $H<H'$.\\
Making use of the theorem on finite increments for the function
$x\mapsto e^{-\theta t^{2x}}$ for $x \in (H,H')$, there exists $\xi \in (H,H')$ such that:
  \begin{eqnarray*}\mathbb{E}\Big(X^{H,K}_{t}-X^{H',K}_{t}\Big)^2&=&
4|H-H'|^2 t^{4\xi}log^2(t)\int^{\infty}_{0}e^{-2\theta t^{2\xi}}\theta^{1-K} d\theta.
\end{eqnarray*}
\begin{itemize}
  \item If $t \leq 1$, then $|t\log(t)|\le e^{-1}$, and:
  \begin{eqnarray*}\mathbb{E}\Big(X^{H,K}_{t}-X^{H',K}_{t}\Big)^2 & \leq & \frac{1}{(e\alpha)^2}|H-H'|^2 \int^{\infty}_{0}e^{-2\theta t^{2\gamma}}\theta^{1-K} d\theta.
 \end{eqnarray*}
  Then \[ \mathbb{E}(X^{H,K}_{t}-X^{H',K}_{t})^2 \leq C_2(\alpha,\gamma,K)|H-H'|^2,\]
 where $$C_2(\alpha,\gamma,K)= \frac{1}{(e\alpha)^2}\sup_{t\in
[a,b]}\bigg(\int^{1}_{0}e^{-2\theta t^{2\gamma}}\theta^{1-K}d\theta+\int^{\infty}_{1}e^{-2\theta t^{2\gamma}}\theta^{1-K} d\theta\bigg)<+\infty.$$
  \item If $t\geq 1$, we obtain:
  \begin{eqnarray*}\mathbb{E}\Big(X^{H,K}_{t}-X^{H',K}_{t}\Big)^2 & \leq & C_3(\alpha,\gamma,K)|H-H'|^2.
 \end{eqnarray*}
where $$C_3(\alpha,\gamma,K)= \bigg[\sup_{t\in
[a,b]}\Big(4t^{4\gamma}log^2(t)\Big)\bigg]\bigg[\sup_{t\in
[a,b]}\bigg(\int^{\infty}_{0}e^{-2\theta t^{2\alpha}}\theta^{1-K} d\theta\bigg)\bigg]<+\infty.$$
\end{itemize}
Finally,
 \begin{eqnarray}\label{X7}\mathbb{E}\Big(X^{H,K}_{t}-X^{H',K}_{t}\Big)^2 & \leq & C_3(\alpha,\gamma,K)|H-H'|^2,
 \end{eqnarray}
 where $$C_3(\alpha,\gamma,K)=\max\Big( C_2(\alpha,\gamma,K)\;;\;C_3(\alpha,\gamma,K)\Big).$$
Consequently, by combining (3) and (4), we conclude
the lemma.
\end{proof}
\begin{remark}
	\begin{enumerate}
		\item A similar result is obtained by Ait Ouahra and Sghir \cite{OS}, (Lemma 3.2), for the sub-fractional Brownian motion $S^H$ with parameter $H\in(0,1)$.  It's a
		continuous centred Gaussian process, starting from zero, with
		covariance function:
		$$\mathbb{E}(S^{H}_{t}S^{H}_{s})=t^{H}+s^{H}-\frac{1}{2}\Big[(t+s)^{H}+|t-s|^{H}\Big].$$
		\item In the case of fBm, (i.e. $K=1$), a similar result is given, independently, in \cite{Boal}, by using the moving average representation of fBm, and in  \cite{LP}, by using the harmonisable representation of fBm.
	\end{enumerate}
\end{remark}
We turn now our interest to the study of the time regularities of our process.
\begin{theorem}\label{regu6}
Let $H(.):[0,\infty)\mapsto [\mu,\nu]\subset (0,1)$ be a H\"{o}lder continuous function with
exponent $\beta>0$ and $\displaystyle\sup_{t\geq0}
H(t)< \beta$. Then for all $t,s\in[0,1],$ there exists a finite positive constant $C(\mu,\nu,K)$ such that:
\[\mathbb{E}\Big(B^{H(t),K}_{t}-B^{H(s),K}_{s}\Big)^2\leq C(\mu,\nu,K) |t-s|^{2(H(t)\vee
H(s))K},\]
where $C(\mu,\nu,K)$ is a  finite positive constant.
\end{theorem}
\begin{proof}[Proof]By the elementary inequality  $(a+b)^2 \leq 2 a^2+2 b^2$, we have:
 \begin{eqnarray*}\mathbb{E}\Big(B^{H(t),K}_{t}-B^{H(s),K}_{s}\Big)^2&=& \mathbb{E}\Big(B^{H(t),K}_{t}-B^{H(t),K}_{s}+B^{H(t),K}_{s}-B^{H(s),K}_{s}\Big)^2\\
 &\leq & 2\mathbb{E}\Big(B^{H(t),K}_{t}-B^{H(t),K}_{s}\Big)^2+2\mathbb{E}\Big(B^{H(t),K}_{s}-B^{H(s),K}_{s}\Big)^2.\end{eqnarray*}
By virtue of (1) and Proposition 2 and the fact that $H(.): [0,+\infty[\rightarrow [\mu,\nu]\subset(0,1)$, we get:
 \begin{eqnarray*}\mathbb{E}\Big(B^{H(t),K}_{t}-B^{H(s),K}_{s}\Big)^2  &\leq &2^{2-K} |t-s|^{2H(t)K}+2C(\mu,\nu,K)|H(t)-H(s)|^{2}.\\
 &\leq &	  2^{2-K} |t-s|^{2H(t)K}+2C'(\mu,\nu,K)|t-s|^{2\beta}.\end{eqnarray*}
Since $\displaystyle\sup_{t\geq 0}H(t)<\beta$  and $K H(t) \in (0,1)$, we
deduce that: $|t-s|^{2\beta}\leq |t-s|^{2KH(t)}$. Thus
\[ \mathbb{E}\Big(B^{H(t),K}_{t}-B^{H(s),K}_{s}\Big)^2\leq C_4(\mu,\nu,K)|t-s|^{2KH(t)},\]
where $C_4(\mu,\nu,K)=2^{2-K}+2C'(\mu,\nu,K)$.\\
Since  the roles of $t$
and $s$ are symmetric, we obtain the desired result.
 \end{proof}
To prove Theorem 2, we need the following classical lemma.
\begin{lemma}
Let $Y$ be a real centred Gaussian random variable. Then for all real $\alpha>0$, we have:\vskip 0.3cm
\hspace*{3.5cm}$\displaystyle\mathbb{E}|Y|^{\alpha}= c(\alpha)\Big(\mathbb{E}|Y|^{2}\Big)^{\frac{\alpha}{2}},$\ \
where \ \ $\displaystyle  c(\alpha)=\frac{2^{\frac{\alpha}{2}}\Gamma(\frac{\alpha+1}{2})}{\Gamma(\frac{1}{2})}.$
\end{lemma}
 \begin{theorem}Let $H(.):[0,\infty)\mapsto [\mu,\nu]\subset (0,1)$ be a H\"{o}lder continuous function with
 	exponent $\beta>0$ and $\displaystyle\sup_{t\geq0}
 	H(t)< \beta$. Then ,there exists $\delta>0$, and for any integer $m\geq1,$ there exist $M_m>0$, such that:
$$\mathbb{E}\Big(B^{H(t),K}_{t}-B^{H(s),K}_{s}\Big)^m\geq M_m|t-s|^{m(H(t)\wedge
H(s))K},$$
for all   $s,t \geq 0$ such that $|t-s|<\delta.$
\end{theorem}
\begin{proof}Using the elementary inequality: $(a+b)^2\geq\frac{1}{2}a^2-b^2,$ we obtain:
 \begin{eqnarray*} \mathbb{E}\Big(B^{H(t),K}_{t}-B^{H(s),K}_{s}\Big)^2 & =& \mathbb{E}\Big(B^{H(t),K}_{t}-B^{H(t),K}_{s}+B^{H(t),K}_{s}-B^{H(s),K}_{s}\Big)^2\\
&\geq & \frac{1}{2}\mathbb{E}\Big(B^{H(t),K}_{t}-B^{H(t),K}_{s}\Big)^2-\mathbb{E}\Big(B^{H(t),K}_{s}-B^{H(s),K}_{s}\Big)^2.
\end{eqnarray*}
Moreover, by using (1) and Proposition 2, we obtain:
\begin{eqnarray*}\mathbb{E}\Big(B^{H(t),K}_{t}-B^{H(s),K}_{s}\Big)^2&\geq & \frac{1}{2^{1+K}}|t-s|^{2H(t)K}-\mathbb{E}\Big(B^{H(t),K}_{s}-B^{H(s),K}_{s}\Big)^2\\
&\geq &  \frac{1}{2^{1+K}}|t-s|^{2H(t)K}- C(\mu,\gamma,K)|H(t)-H(s)|^2.\end{eqnarray*}
Since the function $H(.)$ is H\"{o}lder continuous with exponent $\beta,$ we get:
\begin{eqnarray*} \mathbb{E}\Big(B^{H(t),K}_{t}-B^{H(s),K}_{s}\Big)^2 &\geq & \frac{1}{2^{1+K}}|t-s|^{2H(t)K}- C(\mu,\gamma,K)
|t-s|^{2\beta}\\
&=& |t-s|^{2H(t)K}\Big(\frac{1}{2^{1+K}}-C''(\mu,\gamma,K)|t-s|^{2(\beta-H(t)K)}\Big).
\end{eqnarray*}
Since  $KH(t)< \beta$, we can choose $\delta$ small enough such that for all $s,t\geq0$,
and $|t-s|<\delta$, we have:
$$\frac{1}{2^{1+K}}-C''(\mu,\gamma,K)|t-s|^{2(\beta-H(t)K)}>0.$$
Indeed, it suffices to choose $\displaystyle\delta< \bigg(\frac{1}{2^{1+K}C''(\mu,\gamma,K)}\wedge
1\bigg)^{\eta}$,  where: $\displaystyle \eta=\frac{1}{2}\Big(\beta-K\sup_{t\geq0}
H(t)\Big)^{-1}$.\\ Finally, we get:
$$\mathbb{E}\Big(B^{H(t),K}_{t}-B^{H(s),K}_{s}\Big)^2\geq M
|t-s|^{2H(t)K},\hspace{0.2cm} \text{for all} \hspace{0.1cm}
|t-s|<\delta,$$
where $\displaystyle M=\Big(\frac{1}{2^{1+K}}-C''(\mu,\gamma,K)\delta^{\gamma}\Big).$\\
 Since $B^{H(.),K}$ is a Gaussian process, then by Lemma 1 and the fact that the roles of $t$
 and $s$ are symmetric, we obtain the desired result.
\end{proof}
\begin{remark}It is well known by Berman \cite{Be1} that, for a jointly
	measurable zero-mean Gaussian process $X:= \Big(X(t)\;;\;t\in[0,T]\Big)$ with
	bounded variance, the variance condition:
	$$\int_{0}^{T}\int_{0}^{T}\Big(\mathbb{E}|X_t-X_{s}|^2\Big)^{-\frac{1}{2}}dsdt<+\infty,$$
is sufficient for the local time $L(t,x)$ of $X$ to exist on $[0,T]$ almost surely  and to
be square integrable as a function of $x$:
$$\int_{\mathbb{R}}L^2\Big([a,b],x\Big)dx<+\infty,\hspace*{0.5cm}\Big([a,b]\subset[0,+\infty[\Big).$$
For more informations on local time, the reader is referred to \cite{Boal}, \cite{GH}, \cite{X}  and the references therein. The natural question is to study the local non-determinism property for our process to prove the   joint continuity of local time.  For future work, we plan to study this question. 
\end{remark}
 \section{Local asymptotic self similarity property}
The dependence
of $H(.)$ with respect to the time $t$ destroys all the invariance properties that  we had for the fBm. For example the mBm is no more self-similar, nor with stationary increments. However, the authors in \cite{LP} showed that with the condition that $H(.)$ is
$\beta-$H\"{o}lder continuous with
exponent $\beta>0$ and $\displaystyle\sup_{t\in
\mathbb{R^+}}H(t)<\beta$, the mBm is  locally asymptotically self-similar, (LASS for short), in the following sense\string:
$$ \lim_{\rho\to 0^+}\bigg(\frac{B^{H(t+\rho u) }_{t+\rho u}- B^{H(t)}_t}{\rho^{H(t)}}\;;\;u\geq0\bigg) \stackrel{d}= \Big(\ B^{H(t)}_u\;;\;u\geq0\Big) ,$$
 where $B^{H(t)}_u$ is a fBm with Hurst
parameter $H(t)$, and $\stackrel{d}=$ stands for the convergence of finite dimensional distributions. Some authors use the term localizability for locally asymptotically self-similarity, (see  Falconer \cite{F1},\cite{F2}).\\
Our process is an other example of Gaussian process who loses the self similarity property when $H$ depend on $t$. However, we show in the following result that it is LASS. 
Before we deal with  the proof of our result, we need the
following lemma proved by Ait Ouahra et \textit{al.} \cite{OMS}, (see Theorem 2.6).
\begin{lemma}
	\label{th37} Let $B^{H,K}$ a bfBm with parameters $K\in (0,1)$ and $H\in (0,1)$. Then
		$$\mathbb{E}\bigg(\frac{B^{H,K}_{t+\rho u}-B^{H,K}_t}{h^{HK}}\bigg)^2 \xrightarrow[ \rho\to 0]{}  2^{1-K}.$$
	\end{lemma}
Now, we are ready to state and prove our result.  
\begin{prop}\label{lass77}
Consider $H(.)$ a $\beta$-H\"{o}lder continuous function with exponent $\beta>0$ such that $\displaystyle\sup_{t\geq0}H(t)<\beta$, then $B^{H(.),K}$ is LASS\string:
$$\lim_{\rho\to 0^+}\bigg(\frac{B^{H(t+\rho u),K }_{t+\rho u}- B^{H(t),K}_t}{\rho^{H(t)K}}\;;\;u\geq0\bigg) \stackrel{d}= \Big(\ 2^{1-K}B^{H(t)K}_u\;;\;u\geq0\Big) ,
$$
where $B^{H(t)K}$ is a fBm with the Hurst parameter $H(t)K$.
\end{prop}
 \begin{proof} We use the same arguments used in \cite{LP} in case of the mBm, (see  Proposition 5).\\ We prove the convergence in
distribution by showing the following two statements:
\begin{eqnarray}\label{cent7}
\mathbb{E}\bigg(\frac{B^{H(t+\rho u),K }_{t+\rho u}- B^{H(t),K}_t}{\rho^{H(t)K}}\bigg)
\xrightarrow[{\rho\to 0}]{}0,
\end{eqnarray}
\begin{eqnarray}\label{var7}
\mathbb{E}\bigg(\frac{B^{H(t+\rho u),K }_{t+\rho u}- B^{H(t),K}_t}{\rho^{H(t)K}}\bigg)^2
\xrightarrow[{\rho\to 0}]{}\sigma ^{2}_{t},
\end{eqnarray}
where
$$\sigma
^{2}_{t}=2^{1-K}Var\bigg(\frac{B^{H(t)K}_{t+\rho u}-B^{H(t)K}_t}{\rho^{H(t)K}}\bigg)=2^{1-K},$$
 and $B^{H(t)K}$  is a fBm with the Hurst parameter
$H(t)K.$\\
We dealt with (6) since (\ref{cent7}) is obvious. We  have:
\begin{eqnarray*}\begin{split}&\mathbb{E}\left(\frac{B^{H(t+\rho u),K}_{t+\rho u}-B^{H(t),K}_{t}}{\rho^{H(t)K}}\right)^2\\
&\ \ =\mathbb{E}\left(\frac{B^{H(t+\rho u),K}_{t+\rho u}-B^{H(t),K}_{t+\rho u}}{\rho^{H(t)K}}\right)^2+\mathbb{E}\left(\frac{B^{H(t),K}_{t+\rho u}-B^{H(t),K}_{t}}{\rho^{H(t)K}}\right)^2\\
&\ \ +2\mathbb{E}\left[\frac{\Big(B^{H(t+\rho u),K}_{t+\rho u}-B^{H(t),K}_{t+\rho u}\Big)\Big(B^{H(t),K}_{t+\rho u}-B^{H(t),K}_{t}\Big)}{\rho^{H(t)K}}\right].\end{split}\end{eqnarray*}\\
In view of Proposition 2, and the fact that $H(.)$ is
    $\beta-$H\"{o}lder continuous function, we have:
   \begin{eqnarray*}\mathbb{E}\left(\frac{B^{H(t+\rho u),K}_{t+\rho u}-B^{H(t),K}_{t+\rho u}}{\rho^{H(t)K}}\right)^2 &\leq & C(K) \frac{|H(t+\rho u)-H(t)|^2}{\rho^{2KH(t)}}\\
    &\leq & C'(K) \rho^{2(\beta -KH(t))}.
    \end{eqnarray*}
    Since $\displaystyle K\sup_t H(t)< \beta$, we get:
    $$\mathbb{E}\left(\frac{B^{H(t+\rho u),K}_{t+\rho u}-B^{H(t),K}_{t+\rho u}}{\rho^{H(t)K}}\right)^2\longrightarrow 0 \hspace{0.3cm}\text{as}  \hspace{0.3cm} h\to 0. $$
 In view of Lemma 1, we know that:
    \begin{eqnarray*}\mathbb{E}\left(\frac{B^{H(t),K}_{t+\rho u}-B^{H(t),K}_{t}}{\rho^{H(t)K}}\right)^2\xrightarrow[\rho\to 0]{}2^{1-K}.\end{eqnarray*}
    Now, by Schwartz's inequality, (1) and  Proposition 2, we have:
     \begin{eqnarray*}\begin{split}&\mathbb{E}\left[\frac{\Big(B^{H(t+\rho u),K}_{t+\rho u}-B^{H(t),K}_{t+\rho u}\Big)\Big(B^{H(t),K}_{t+\rho u}-B^{H(t),K}_{t}\Big)}{h^{H(t)K}}\right]\\
     &\ \ \leq \left[\mathbb{E}\left(\frac{B^{H(t+\rho u),K}_{t+\rho u}-B^{H(t),K}_{t+\rho u}}{\rho^{H(t)K}}\right)^2\right]^{\frac{1}{2}} \left[\mathbb{E}\left(\frac{B^{H(t),K}_{t+\rho u}-B^{H(t),K}_{t}}{\rho^{H(t)K}}\right)^2\right]^{\frac{1}{2}}\\
     &\ \ \leq  C \rho^{\beta-H(t)K}\rightarrow 0,\ \ \mbox{since}\;  K\sup_t H(t)< \beta.\end{split}\end{eqnarray*}
 Hence, we deduce that:
\[\mathbb{E}\left(\frac{B^{H(t+\rho u),K}_{t+\rho u}-B^{H(t),K}_t}{\rho^{H(t)K}}\right)^2
\xrightarrow[\rho\to 0]{} 2^{1-K}.\]
Consequently the LASS property is proved.
\end{proof}
\section{Long range dependence}
The long range dependence, (LRD for short), and long memory
are synonymous notions. LRD measures long-term correlated
processes. LRD is a
characteristic of phenomena whose autocorrelation functions  decay
rather slowly. The presence and the extent of LRD  is usually
measured by the parameters of the process. Most of the definitions of LRD appearing in literature for stationary process
are based on the second-order properties of a stochastic process.
Such properties include asymptotic behavior of covariances, spectral
density, and variances of partial sums. The specialness of
LRD is a connection between long memory and  stationarity, (see for example \cite{ST} in case of fBm and \cite{MT} and \cite{RT} in case of bfBm).  For our process $B^{H(.),K}$, we use the same arguments used in \cite{Aal} in case of standard mBm. Of course, the definitions must be adapted in  our case since  mBm and our extension does not  have stationary increments, (see for example \cite{Aal}):
\begin{definition}
\begin{enumerate}
	\item[a)] Let $Y$ be a second order process. $Y$ is said to have a LRD if there exists a function $\alpha(s)$ taking values in $(-1,0)$ such that:
	$$\forall s\geq 0,\quad cor_Y(s,s+h)\thickapprox h^{\alpha(s)}\quad \mbox{as h tends to} \;+\infty,$$
		where $\displaystyle cor_Y(s,t):=\frac{cov_Y(s,t)}{\sqrt{\mathbb{E}(Y^2(s))\mathbb{E}(Y^2(t))}} $.
	\item [b)]  Let $Y$ be a second-order process. $Y$ is said to have a LRD if:
	\[\forall s\geq 0,\;\forall \delta\geq 0,\ \ \sum^{+\infty}_{0}|cor_Y(s,s+k\delta)|=+\infty.\] 
\end{enumerate}	
\end{definition}
In the next propositions, we prove some results about covariances and correlations of our process and its increments. In the sequel, we denote $f(t)\approx g(t)$ when $t$ if there  exist $0<c<d<+\infty$  such that for all sufficiently large $t$: $\displaystyle c\le \frac{f(t)}{g(t)}\le d$. We put: $cov(t,s):=R^{H(.),K}(t,s)$ the covariance function of our process $B^{H(.),K}$ and $cor(t,s)$  its correlation function.   \vskip 0.3cm
\textbf{A) Asymptotic behaviour of the covariance and the correlation  of $B^{H(.),K}$:}
\begin{prop}
When $t$ tends to infinity, and for all fixed $s\geq 0,$ we have:
\begin{enumerate}
	\item[i)] $K(H(t)+H(s))<1\Rightarrow cov(t,s)\thickapprox t^{(H(t)+H(s))(K-1)}.$
\item[ii)] $K(H(t)+H(s))>1\Rightarrow cov(t,s)\thickapprox t^{K(H(t)+H(s))-1}.$
\item[iii)] $K(H(t)+H(s))<1\Rightarrow cor(t,s)\thickapprox t^{-H(t)}.$
\item[iv)] $K(H(t)+H(s))>1\Rightarrow cor(t,s)\thickapprox t^{KH(s)-1}. $
\end{enumerate}
\end{prop}
\begin{proof}[Proof]\begin{enumerate}
		\item[i)] and ii) follows from a Taylor expansion of: \[\Big(t^{H(t)+H(s)}+s^{H(t)+H(s)}\Big)^K-|t-s|^{(H(t)+H(s))K},\] we obtain:
\[cov(t,s)=R^{H(.),K}(t,s)\thickapprox K(D(H(t),H(s))^K\Big[s^{H(t)+H(s)}t^{(H(t)+H(s))(K-1)}\]
\[+(H(t)+H(s))st^{K(H(t)+H(s))-1}\Big],\ \ \mbox{as}\;t \to \infty,\]
where the leading term is:
\[K(D(H(t),H(s))^K s^{H(t)+H(s)}t^{(H(t)+H(s))(K-1)}\quad \text{if}\quad K(H(t)+H(s))<1,\]
and
\[K(D(H(t),H(s))^K (H(t)+H(s))st^{K(H(t)+H(s))-1}\quad\text{if}\quad K(H(t)+H(s))>1.\]
(Recall that $H(t)+H(s)$ and $(D(H(t),H(s))^K$ are bounded).
	\item[iii)] and vi): Using once again a Taylor expansion of:
\[cor(t,s):=(2D(H(t),H(s))^K\frac{\Big(t^{H(t)+H(s)}+s^{H(t)+H(s)}\Big)^K-|t-s|^{(H(t)+H(s))K}}{s^{KH(s)}t^{KH(t)}},\]
where the leading term in this case is:
\[ K(D(H(t),H(s))^Ks^{H(t)+(1-K)H(s)}t^{-H(t)+(K-1)H(s)}\quad\text{if}\quad K(H(t)+H(s))<1,\]
and
\[  K(D(H(t),H(s))^K(H(t)+H(s))t^{KH(s)-1}s^{1-KH(s)}\quad\text{if}\quad K(H(t)+H(s))>1.\]
\end{enumerate}
\end{proof}
Since both $-H(t)$ and $KH(s)-1$ belong to $(-1,0)$ for all $t,s,$ we have the following  result:
\begin{corollary}
For all admissible $H(t)$, our process $B^{H(.),K}$ has LRD in the sense of Definition 2)b). If, for all $s$, $K(H(t)+H(s))>1$ for all sufficiently large $t,$  then $B^{H(.),K}$ has LRD in the sense of Definition 2)a), with functional LRD exponent: $\alpha(s)=KH(s)-1.$
\end{corollary}
\textbf{B) Asymptotic behaviour of  the covariance and the correlation  of the increments of $B^{H(.),K}$:}\vskip 0.3cm
In the following results, to simplify the notation, let us denote:
\begin{equation*}\begin{split}\textstyle &L(s,t):=\max\big(H(t) + H(s),H(t + 1) + H(s), H(t) + H(s + 1), H(t + 1) + H(s + 1)\big)\end{split}\end{equation*}
\begin{prop}Let $Y$ be the
unit time increments of $B^{H(.),K}$:$\ \ Y (t) = B^{H(t+1),K}(t+1)-B^{H(t),K}_{t}.$ Then, when $t$ tends to
infinity, and for all fixed $s\geq 0$  such that the four quantities: $H(t) + H(s),$ $ H(t +
1) + H(s),$  $ H(t) + H(s + 1),$ and $ H(t + 1) + H(s + 1)$ are all different, we have:
\begin{enumerate}
	\item [i)]  $KL(s,t) < 1 \Rightarrow \quad cov_Y (t,s)\thickapprox t^{L(s,t)(K-1)}.$
	\item [ii)] $KL(s,t) > 1 \Rightarrow \quad cov_Y (t,s) \thickapprox t^{KL(s,t)-1}.$
	\item [iii)] $\textstyle KL(s,t)< 1
	\Rightarrow \quad cor_Y (t,s) \thickapprox t^{-K\max(H(t),H(t+1))}.$
	\item [iv)] $\textstyle KL(s,t) > 1\Rightarrow \quad cor_Y (t,s) \thickapprox t^{K\max(H(s),H(s+1))-1}.$
\end{enumerate}
\end{prop}
\begin{proof}\begin{enumerate}
		\item [i)] and ii): By definition, we have:
\[cov_Y(t,s)=cov(t+1,s+1)-cov(t+1,s)-cov(t,s+1)+cov(t,s).\]
Applying the Taylor expansion to each covariance, we obtain:
\begin{itemize}
  \item if $KL(s,t) <1$, from Proposition 4, it follow that:
  \[ cov_Y (t,s)\thickapprox t^{L(s,t)(K-1)}.\]
  \item if at least one of $K(H(t) + H(s));$  $K(H(t + 1) + H(s));$  $K(H(t) + H(s + 1))$ and $K(H(t + 1) + H(s + 1)))$ is greater than one, the order of of $cov_Y(t,s)$ will be the maximum of these value, since they all differ. More
precisely, denoting $(t', s')$ the couple where the maximum of $H(t) + H(s);$  $H(t + 1) + H(s);$  $H(t) + H(s + 1)$ and $H(t + 1) + H(s + 1)$ is attained, we get:\\
\begin{equation*}\begin{split}
cov_Y(t,s)=K(D(H(t),H(s))^K(H(t')+H(s'))s't'^{K(H(t')+H(s'))-1}+o(t'^{K(H(t')+H(s'))-1}).
\end{split}\end{equation*}
\end{itemize}
	\item [iii)] and iv): Again, this is simply obtained using Proposition 5 and the fact that $E(Y^2(t))=O(t^{2Kmax(H(t),H(t+1))})$ if $H(t)$ and $H(t + 1)$ differ (otherwise cancellation occur
	and the leading term is different). The exponent in the case where $ KL(s,t) > 1$ results from
	the identity :
	\begin{equation*}\begin{split}&\max\big(H(t) + H(s), H(t + 1) + H(s), H(t) + H(s + 1), H(t + 1) + H(s + 1)\big)\\&-\max\big(H(t); H(t + 1)\big) = max(H(s), H(s + 1)).\end{split}\end{equation*}
\end{enumerate}
\end{proof}
\begin{corollary}For all admissible $H(t)$, our process $B^{H(.),K}$ has LRD in the sense of Definition 2)b). If, for all $s$, $ KL(s,t) > 1$, for all sufficiently large $t$,  the increments of $B^{H(.),K}$ have LRD in the sense of Definition 2. As well as in the sense of Definition 2)a), with functional long range
dependence exponent $\alpha(s) = K\max(H(s),H(s + 1)) -1.$
\end{corollary}
\begin{proof}[Proof]
Obviously, both $K\max(H(s),H(s + 1))-1$ and $-K\max(H(t),H(t + 1))$ belong to
$(-1; 0).$
\end{proof}
\textbf{Conclusion and Outlook:}
\begin{enumerate}
  \item [(i)] If we can prove the local non-determinism property for our process, (see Berman \cite{Be2}),  then Theorem 2 will be interesting to prove the existence and the H\"older regularities of the local time of our process, (see \cite{Boal} in case of the mBm).
  \item [(ii)] A response of the problem of the decomposition in law of our process appeared in Remark 2 will be useful to generalize a popular results for the bfBm like the existence and the H\"older regularities of its local time, (see  \cite{Oal} in case of bfBm and \cite{Boal} in case of mBm).
  \item [(iii)] It will be interesting to study a general case of Gaussian process of the form $B^{H(.),K(.)}$ where both the parameters $H$ and $K$ depend on the time $t$.
\end{enumerate}

\begin{thebibliography}{99}
	\bibitem {AN} P. S. Addison and A. S. Ndumu (1999). Engineering applications of fractional Brownian motion: self-affine and self-similar
random processes. Fractals, 07(2), p. 151-157.
	\bibitem{OMS} M. Ait Ouahra, A. Sghir and S. Moussaten (2017). On limit theorems of some
extensions of fractional Brownian
motion and their additive
functional. Stochastic and
Dynamics. 17, No. 3, (14 pages).
	\bibitem{Oal} M. Ait Ouahra, H. Ouahhabi and A. Sghir (2019). Continuity in law of some additive
functionals of the bifractional
Brownian motion. Stochastics an
international journal of probability and
stochastic processes. Stochastics, 1–16. doi:10.1080/17442508.2019.1568436.
	\bibitem{OS} M. Ait Ouahra and A. Sghir (2017). Continuity in law of some additive functionals of the	subfractional Brownian motion. Stochastic Analysis and Applications. Vol00, p, 1-14.
	\bibitem{Aal} A. Ayache, S. Cohen and J. L\'{e}vy Véhel (2002). The covariance structure of multifractional
	Brownian motion, with application to long range dependence, Proceedings of ICASSP,
	Istambul.
	\bibitem{Beal} A. Benassi, S. Jaffard and D. Roux (1997). Elliptic Gaussian random processes, Rev. Mat.
	Iberoamericana 13, p. 19-90.
	\bibitem{Be1} S. M. Berman (1969).  Local times and sample function properties of stationary	Gaussian processes, Trans. Amer. Math. Soc. 137, p.	277-299.
	\bibitem{Be2} S. M. Berman (1973). Local nondeterminism and local times of gaussian processes.
		Indiana. Univ. Math. J. 23, p. 69-94.
	\bibitem {Boal} B. Boufoussi, M. Dozzi and R. Guerbaz (2006). On the local time of multifractional Brownian motion. Stochastics: An international Journals of Probability and Stochastic Process 78(1), p. 33–49. MR0239652.
	\bibitem {Ch}  P. Cheridito (2001). Mixed fractional Brownian motion. Bernoulli, 7(6), p. 913-934.
	\bibitem {Co} S. Cohen (1999). From self-similarity to local self-similarity: the estimation problem, Fractals:
	Theory and Applications in Engineering. M. Dekking, J. L\'{e}vy Véhel, E. Lutton and C.
	Tricot (Eds), Springer Verlag, 3 - 16.
	\bibitem {Cal}  A. Comegna, A. Coppola, V. Comegna, A. Sommella  and C.D. Vitale (2013). Use of a fractional Brownian motion model to mimic spatial
	horizontal variation of soil physical and hydraulic properties displaying a power-law variogram.
	Procedia  Environmental Sciences, 19, p. 416–425.
	\bibitem {ET}  K. Es-Sebaiy and C. A. Tudor (2007). Multidimensional bifractional Brownian Motion. It\^{o} and
	Tanaka formulas. Stochastics and Dynamics, 7(3), p. 365–388.
	\bibitem{F1} K.J. Falconer (2002). Tangent fields and the local structure of random fields. Journal of Theoretical Probability. 15(3), p. 731–750.
	\bibitem{F2} K.J. Falconer (2003). The local structure of random processes. Journal of the London Mathematical Society. 67(3), p. 657-672.
	\bibitem{GH} D. Geman and J. Horowitz (1980). Occupation densities, Ann. Probab.
	8(1), p. 1-67.
	\bibitem{HV} C. Houdré and J. Villa (2003). An example of infinite dimensional quasi-helix, Contemp. Math. (Amer.
	Math. Soc). 336, p. 195–201.
		\bibitem {Ka}    J. P. Kahane (1981).  Hélices et quasi-hélices. Advances in Mathematics. 7B, p. 417–433.
		\bibitem{Ko} A. N. Kolmogorov (1940). The Wiener spiral and some other interesting curves in Hilbert space, Dokl. Akad. Nauk SSSR, vol.
		26, no. 2, p. 115–118.
	\bibitem {LN}  P. Lei and  D. Nualart (2009). A decomposition of the bifractional Brownian motion and some
	applications. Statistics and Probabilty Letters. 79,  p. 619-624.
	\bibitem{LP} J. L\'{e}vy-Vehel and R. F. Peltier (1996). Multifractional Brownian motion: definition and
	preliminary results, Techn. Report RR-2645, INRIA.
	\bibitem {MT} M. Maejima and C. A. Tudor (2008). Limits of bifractional Brownian noises. Communications on Stochastic Analysis, 2(3), p. 369-383.
\bibitem {MVN} B.	Mandelbrot  and J.W. Van Ness (1968). Fractional Brownian Motions, Fractional Noises and Applications. SIAM Review, 10, 422-437.
\bibitem {RT}  F.  Russo and C. Tudor (2006). On the bifractional Brownian motion. Stochastic Processes and their Applications, 116,  p. 830–856. 
\bibitem{ST} G. Samorodnitsky and M. Taqqu (1994). Stable Non-Gaussian Random Processes, Chapman and Hall, New York.
	\bibitem{MTa}  M. S. Taqqu (2002). The modelling of ethernet data and of signals that are heavy-tailed with infinite variance.
	Scandinavian Journal of Statistics, 29(2), p. 273-295.
\bibitem{X}     Y. Xiao (1997). H\"{o}lder conditions for the local times and the Hausdorff measure of the level sets
	of Gaussian random fields. Probab. Theorey Related Fields. 109(1), p. 129-157.
	\end{thebibliography}
\end{document}